\documentclass[12pt]{amsart}
\usepackage{graphics}
\usepackage{amsfonts,amssymb,color}
\usepackage[mathscr]{eucal}
\usepackage{amsmath, amsthm}
\usepackage{mathrsfs}
\usepackage{amsbsy}
\usepackage{wasysym}
\usepackage{url}

\input xypic
\xyoption{all}

\makeindex
\makeglossary

\begin{document}
\baselineskip = 16pt

\newcommand \ZZ {{\mathbb Z}}
\newcommand \NN {{\mathbb N}}
\newcommand \RR {{\mathbb R}}
\newcommand \PR {{\mathbb P}}
\newcommand \AF {{\mathbb A}}
\newcommand \GG {{\mathbb G}}
\newcommand \QQ {{\mathbb Q}}
\newcommand \CC {{\mathbb C}}
\newcommand \bcA {{\mathscr A}}
\newcommand \bcC {{\mathscr C}}
\newcommand \bcD {{\mathscr D}}
\newcommand \bcF {{\mathscr F}}
\newcommand \bcG {{\mathscr G}}
\newcommand \bcH {{\mathscr H}}
\newcommand \bcM {{\mathscr M}}
\newcommand \bcJ {{\mathscr J}}
\newcommand \bcL {{\mathscr L}}
\newcommand \bcO {{\mathscr O}}
\newcommand \bcP {{\mathscr P}}
\newcommand \bcQ {{\mathscr Q}}
\newcommand \bcR {{\mathscr R}}
\newcommand \bcS {{\mathscr S}}
\newcommand \bcV {{\mathscr V}}
\newcommand \bcW {{\mathscr W}}
\newcommand \bcX {{\mathscr X}}
\newcommand \bcY {{\mathscr Y}}
\newcommand \bcZ {{\mathscr Z}}
\newcommand \goa {{\mathfrak a}}
\newcommand \gob {{\mathfrak b}}
\newcommand \goc {{\mathfrak c}}
\newcommand \gom {{\mathfrak m}}
\newcommand \gon {{\mathfrak n}}
\newcommand \gop {{\mathfrak p}}
\newcommand \goq {{\mathfrak q}}
\newcommand \goQ {{\mathfrak Q}}
\newcommand \goP {{\mathfrak P}}
\newcommand \goM {{\mathfrak M}}
\newcommand \goN {{\mathfrak N}}
\newcommand \uno {{\mathbbm 1}}
\newcommand \Le {{\mathbbm L}}
\newcommand \Spec {{\rm {Spec}}}
\newcommand \Gr {{\rm {Gr}}}
\newcommand \Pic {{\rm {Pic}}}
\newcommand \Jac {{{J}}}
\newcommand \Alb {{\rm {Alb}}}
\newcommand \Corr {{Corr}}
\newcommand \Chow {{\mathscr C}}
\newcommand \Sym {{\rm {Sym}}}
\newcommand \Prym {{\rm {Prym}}}
\newcommand \cha {{\rm {char}}}
\newcommand \eff {{\rm {eff}}}
\newcommand \tr {{\rm {tr}}}
\newcommand \Tr {{\rm {Tr}}}
\newcommand \pr {{\rm {pr}}}
\newcommand \ev {{\it {ev}}}
\newcommand \cl {{\rm {cl}}}
\newcommand \interior {{\rm {Int}}}
\newcommand \sep {{\rm {sep}}}
\newcommand \td {{\rm {tdeg}}}
\newcommand \alg {{\rm {alg}}}
\newcommand \im {{\rm im}}
\newcommand \gr {{\rm {gr}}}
\newcommand \op {{\rm op}}
\newcommand \Hom {{\rm Hom}}
\newcommand \Hilb {{\rm Hilb}}
\newcommand \Sch {{\mathscr S\! }{\it ch}}
\newcommand \cHilb {{\mathscr H\! }{\it ilb}}
\newcommand \cHom {{\mathscr H\! }{\it om}}
\newcommand \colim {{{\rm colim}\, }} 
\newcommand \End {{\rm {End}}}
\newcommand \coker {{\rm {coker}}}
\newcommand \id {{\rm {id}}}
\newcommand \van {{\rm {van}}}
\newcommand \spc {{\rm {sp}}}
\newcommand \Ob {{\rm Ob}}
\newcommand \Aut {{\rm Aut}}
\newcommand \cor {{\rm {cor}}}
\newcommand \Cor {{\it {Corr}}}
\newcommand \res {{\rm {res}}}
\newcommand \red {{\rm{red}}}
\newcommand \Gal {{\rm {Gal}}}
\newcommand \PGL {{\rm {PGL}}}
\newcommand \Bl {{\rm {Bl}}}
\newcommand \Sing {{\rm {Sing}}}
\newcommand \spn {{\rm {span}}}
\newcommand \Nm {{\rm {Nm}}}
\newcommand \inv {{\rm {inv}}}
\newcommand \codim {{\rm {codim}}}
\newcommand \Div{{\rm{Div}}}
\newcommand \CH{{\rm{CH}}}
\newcommand \sg {{\Sigma }}
\newcommand \DM {{\sf DM}}
\newcommand \Gm {{{\mathbb G}_{\rm m}}}
\newcommand \tame {\rm {tame }}
\newcommand \znak {{\natural }}
\newcommand \lra {\longrightarrow}
\newcommand \hra {\hookrightarrow}
\newcommand \rra {\rightrightarrows}
\newcommand \ord {{\rm {ord}}}
\newcommand \Rat {{\mathscr Rat}}
\newcommand \rd {{\rm {red}}}
\newcommand \bSpec {{\bf {Spec}}}
\newcommand \Proj {{\rm {Proj}}}
\newcommand \pdiv {{\rm {div}}}
\newcommand \wt {\widetilde }
\newcommand \ac {\acute }
\newcommand \ch {\check }
\newcommand \ol {\overline }
\newcommand \Th {\Theta}
\newcommand \cAb {{\mathscr A\! }{\it b}}

\newenvironment{pf}{\par\noindent{\em Proof}.}{\hfill\framebox(6,6)
\par\medskip}

\newtheorem{theorem}[subsection]{Theorem}
\newtheorem{conjecture}[subsection]{Conjecture}
\newtheorem{proposition}[subsection]{Proposition}
\newtheorem{lemma}[subsection]{Lemma}
\newtheorem{remark}[subsection]{Remark}
\newtheorem{remarks}[subsection]{Remarks}
\newtheorem{definition}[subsection]{Definition}
\newtheorem{corollary}[subsection]{Corollary}
\newtheorem{example}[subsection]{Example}
\newtheorem{examples}[subsection]{examples}

\title{A remark on Algebraic cycles on cubic fourfolds}
\author{Kalyan Banerjee}

\address{Harish Chandra Research Institute, India}

\email{banerjeekalyan@hri.res.in}

\begin{abstract}
In this short note we try to generalize the Clemens-Griffiths criterion of non-rationality for smooth cubic threefolds to the case of smooth cubic fourfolds.
\end{abstract}
\maketitle

\section{Introduction}
In algebraic geometry, one of the most important problems is to detect the rationality of a given variety. In the remarkable paper \cite{CG} by Clemens and Griffiths it has been proved that a smooth cubic threefold is non-rational. The method was elegant. They proved that a rational, smooth threefold must have its intermediate jacobian isomorphic to product of Jacobian of curves. Further it has been proved that no smooth cubic threefold satisfies this criterion. The next question is a smooth cubic fourfold rational, at least a very general one. Recent work by Hassett suggests that there exists smooth cubic fourfolds which are rational, \cite{H}. In this paper we propose a criterion similar to that of Clemens and Griffiths, by studying the group of algebraically trivial one cycles modulo rational equivalence(denoted by $A_1$) of the cubic fourfold.

\medskip

We first prove that the group of algebraically trivial one cycles on a cubic fourfold is of essential dimension $2$, meaning that $A_1$ of a cubic fourfold admits a surjective map from $A_0$ (algebraically trivial zero cycles modulo rational equivalence) of a smooth projective surface, and that it is not dominated by the Jacobian of a smooth projective curve, this later fact is due to Schoen \cite{Sc}. Then we prove that $A_1$ of a cubic fourfold is isomorphic to the kernel of the push-forward at the level of $A_0$ induced by a correspondence between two surfaces, which occur in a very natural way. At the same time we prove that the kernel of the above homomorphism at the level of $A_0$ cannot admit a surjective map from the $A_0$ of some other smooth, projective, surface. This point is subtle. Let $X$  be the cubic fourfold and $T,S$ are the surfaces and $\Gamma$ is a correspondence on $T\times S$. Let us consider $\ker(\Gamma_*)$ from $A_0(T)$ to $A_0(S)$. Then we prove that there does not exists a surface $S'$, and correspondences $\Gamma_1$ supported on $S'\times T$, and $\Gamma_2$ supported on $S'\times X$, such that image of $\Gamma_{1*}$ is the kernel of $\Gamma_*$ and the diagram

$$
  \diagram
 A_0(S')\ar[dd]_-{} \ar[rr]^-{\Gamma_{2*}} & & A_1(X)) \ar[dd]^-{} \\ \\
  A_0(S') \ar[rr]^-{\Gamma_{1*}} & & \ker(A_0(T)\to A_0(S))
  \enddiagram
  $$
commutes.

\begin{theorem}
$A_1$ of a smooth cubic fourfold  is isomorphic to the kernel of the push-forward induced by a correspondence between two fixed smooth projective surfaces and it is not dominated by the $A_0$ of a single smooth projective surface (in the sense mentioned above).
\end{theorem}

\smallskip

This raises a natural criterion for non-rationality of a cubic fourfold interms of $A_0$ of smooth projective surfaces. We define that the $A_1$ of a cubic fourfold is representable upto dimension two if it is dominated by a sum of $A_0$'s of smooth projective curves and smooth projective surfaces. It follows that if a cubic is rational then it is representable upto dimension $2$ and further the kernel of the homomorphism from the finite direct sum of $A_0$ of smooth projective surfaces and curves to the $A_1$ of the cubic is isomorphic to a finite direct sum of smooth projective surfaces and curves . Since any smooth cubic fourfold is unirational the group $A_1$ is always dominated by a finite direct sum of $A_0$ of surfaces and curves. But suppose that the kernel of this homomorphism from a finite direct sum of $A_0$'s of surfaces and curves to the group $A_1$ of the cubic is not isomorphic to a finite direct sum of $A_0$'s of surfaces and curves, then it is non-rational. The main result of this paper is as follows:

\smallskip

So the natural question is whether for a very general smooth cubic fourfold the the kernel of the homomorphism from a finite sum of $A_0$'s of smooth projective surfaces and curves to $A_1$ of the cubic  is not isomorphic to a finite sum of $A_0$'s of smooth projective surfaces and curves, or can we say that the smooth cubic fourfolds satisfying the criterion that the above kernel being  isomorphic to a finite sum of $A_0$'s of smooth projective surfaces  and curves are parametrized by a countable union of the Zariski closed subsets in  the moduli of cubic fourfolds.

\medskip
In  this context we would like to mention that recently C.Voisin has an approach to the non-rationality of a smooth cubic fourfold in terms of Chow theoretic decomposition of the diagonal of the cubic, \cite{CP}, \cite{V1},\cite{V2}. According to that approach if a smooth cubic fourfold does not admit the Chow theoretic decomposition of the diagonal then its not stable rational, hence non-rational. We are investigating the relations between our approach and that of Voisin. This issue will be dealt in a sequel of articles.

{\small \textbf{Acknowledgements:} The author thanks Kapil Paranjape for his constant encouragement and for carefully listening about the arguments of this paper from the author. The author thanks R.Laterveer for a helpful discussion related to the theme of this paper.}

Throughout this text we work over the field of complex numbers.

\section{Conic bundle construction and the essential dimension of Chow group of a cubic fourfold}
\label{section1}

We work over the field of complex numbers. Let $X$ be a cubic fourfold. Let $L$ be a line on $X$ (since $X$ is unirational, such a line exists). Then we project $X$ onto $\PR^3$ from the line $L$, that is we consider the rational map induced by the line $L$. Then we blow up the indeterminacy locus of this rational map, to get $X_L$, and a regular map from $X_L$ to $\PR^3$. The inverse image of a point in $\PR^3$ is a conic in $\PR^2$. Consider the surface $S$ in $\PR^3$ such that the inverse image of a point is a union of two lines. This surface is called the discriminant surface. Let $T$ be the double cover of $S$ sitting inside $F(X)$. Then we prove the following:

\begin{theorem}
The map from $A_0(T)$ to $A_1(X)$ is onto, and the composition $Z_*f^*$ sends $A_0(S)$ to zero under this composition.
\end{theorem}

\begin{proof}

Consider a hyperplane section $X_t$ of $X$, then $L$ might or might not be contained in $X_t$. If $L$ is not contained in $X_t$, then we have a regular map from $X_t$ to $\PR^2$ (by considering successive projections). Let us consider two lines on $X_t$, then under the projection $\pi_{tL}$ from $L$, they are mapped onto two rational curves on $\PR^2$, by Bezout's theorem these curves intersect and we get a point $z$ in $\PR^2$. Since the inverse image of $z$ is a union of two lines, $z$ belongs to the discriminant curve $S_t$ of the projection.

Now let $L$ belong to $X_t$. Then we have to prove that for pairs of lines of the form $(L,L')$, there exists a point $z$ on the discriminant curve. For that we consider the strict transform of $X_t$ along $L$. Then we have a regular map from $X_{tL}$ to $\PR^2$. Under this blow up we have  decomposition of $A^2(X_{tL})\cong \pi_{tL}^*(A^2(X_t))\oplus j_*(A^1(Z))$, where $Z$ is a $\PR^1$-bundle over $L$. This means that $\pi_{tL}^{-1}(L)$ is a ruled surface, which produces a divisor on $X_t$, supported on $Z$, since $A^1(Z)$ is trivial, this divisor is rationally trivial. Hence the group $A^2(X_{tL})$, and $A^2(X_t)$ is generated by differences of lines of the form $L_1-L_2$, where $L_1,L_2$ are different from $L$. For such differences we have a point $z$ on the discriminant curve $S_t$.  Then this analysis tells us that the composition of $Z_{t*}$ is surjective from $A_0(T_t)$ to $A_1(X_t)$, where $Z_t$ is the Universal line on the product $X_t\times F(X_t)$ and $f_t$ is the $2:1$ map from $T_t\to S_t$. Then if we compose $Z_{t*}$ and $f_t^*$, it will give $\pi_L^*$, so we have $A_0(S_t)$ goes to zero under $Z_{t*}\circ f_t^*$. Now $A_0(S)$ is generated by $A_0(S_t)$, where $S_t$ is a smooth projective curve in $S$, on the other hand $A_1(X)$ is generated by $A_1(X_t)$, where $X_t$ is a smooth hyperplane section, by Bertini's theorem. So we get the result.

\end{proof}
Now we try to understand the kernel of the push-forward from $A_0(T_t)$ to $A_0(T)$ and prove the following:

\begin{theorem}
For a general $t$ the kernel of $A_0(T_t)\to A_0(T)$ is  contained in the kernel of $A_0(T_t)\to A_1(X_t)$.
\end{theorem}

\begin{proof}

Consider the following commutative diagram.

$$
  \diagram
   \Sym^g T_t\ar[dd]_-{} \ar[rr]^-{} & & \Sym^g T \ar[dd]^-{} \\ \\
  A_0(T_t) \ar[rr]^-{j_{t*}} & & A_0(T)
  \enddiagram
  $$
Since the fiber of the right vertical map is a countable union of Zariski closed subsets in $\Sym^g T$, and the map from $\Sym^g T_t$ to $A_0(T_t)$ is onto, we have that the kernel of $j_{t*}$ is a countable union of Zariski closed subsets in $J(T_t)$, under the identification $A_0(T_t)\cong J(T_t)$. Since $\ker(j_{t*})$ is a subgroup and we work over uncountable ground fields, we have the kernel is a countable union of shifts of an abelian subvariety $A_t$ inside $J(T_t)$ [for more details please see, \cite{BG}].

Now we prove that this kernel is either countable or all of $J(T_t)$ by using monodromy argument. First of all we notice that $Z_{t*}$ is onto from $A_0(T_t)$ to $A_1(X_t)$, which by the equivalence of Hodge structures and abelian varieties gives rise to a surjective map $Z_{t*}$ from $H^1(T_t,\QQ)$ to $H^3(X_t,\QQ)$. Considering a Lefschtez pencil through $X_t$, we have an action of $\pi_1(\PR^1\setminus \{0_1,\cdots,0_m\},t)$ on $H^3(X_t,\QQ)$, hence on $Z_*(H^1(T_t,\QQ))$. This action is given by the Picard Lefschetz formula and it acts irreducibly on $H^3(X_t,\QQ)$ \cite{Voisin},\cite{Voi}, and hence on the image of $Z_*$.

Now suppose that there exists $t$ such that $A_t$ is neither zero nor all of $J(T_t)$. Then this $t$ corresponds to a morphism $\CC[s]\to \CC$, and using this morphism we can spread $A_t,J(S_t)$ to all over $\Spec(\CC[s])$. This can be done as follows. Suppose that $A_t$ is given by $\CC[x,y,z]/f(x,y,z)$, the we can consider its spread over $\CC[t]$ to be $\Spec(\CC[s][x,y,z])/f(x,y,z)$. Then there exists $\bcA$ and $\bcJ$, which are spreads of $A_t,J(T_t)$ respectively. As soon as we have a spread we can consider the fibration given by them over some Zariski open subset of $\mathbb A^1$. Any such fibration corresponds to a locally constant sheaf and hence action of the fundamental group of $\pi_1(U,t)$ on the stalks of the locally constant sheaf, where $U$ is Zariski open around $t$. The point $t$ belongs to $U$, because the fiber over it is smooth (so when applying Ehressmann's theorem on fibration, we don't throw away $t$). Therefore $\pi_1(U,t)$
acts on $Z_*(H^{2d-1}(A_t,\QQ))$, where $d$ is the dimension of $A_t$. On the other hand $Z_*(H^{2d-1}(A_t,\QQ))$ is embedded in $Z_*(H^1(T_t,\QQ))$ and this embedding is a map of $\pi_1(U,t)$ modules (because it is induced by a regular morphism of algebraic varieties). Therefore $Z_*(H^{2d-1}(A_t,\QQ))$ is $\pi_1(U,t)$ invariant and hence it is equal to $\{0\}$ or $Z_*(H^1(T_t),\QQ)$. So either $A_t$ is contained inside the kernel of $J(T_t)\to A_1(X_t)$ or it maps surjectively onto $A_1(X_t)$. The later cannot happen because then $A_1(X_t)$ maps to zero for a general $t$ into $A_1(X)$, which gives that $A_1(X)$ is representable, which is not true. Hence for a general $t$, we have that  $A_t$ is contained inside the kernel of $J(T_t)\to A_1(X_t)$.
\end{proof}

\begin{remark}
This arguments goes through for any uncountable ground field of characteristic zero, we have to argue by using \'etale fundamental groups instead of usual fundamental groups.
\end{remark}

\section{Prym construction of correspondences}
In this section we are interested in the following problem. Let $\Gamma_1$ be a correspondence on $S_1\times X$, where $S_1$ is a surface and $X$ a cubic fourfold, such that $\Gamma_{1*}$ is an isomorphism from $A_0(S_1)$ to $A_1(X)$. Similarly let $\Gamma_2$ be a correspondence on $S_2\times X$, where $S_2$ is another surface, such that $\Gamma_{2*}$ is an isomorphism. Then can we cook up $\Gamma$, which is a correspondence of degree zero between $S_1$ and $S_2$. So we prove the following theorem:

\begin{theorem}
Let $S_1,S_2$ be smooth projective surfaces and $X$ is a smooth  fourfold. Let $\Gamma_1,\Gamma_2$ are correspondences from $S_1,S_2$ to $X$. Let $\Gamma_i'$ be the extension of $\Gamma_i$ supported on $S_{i\CC(S_i)}\times X_{\CC(S_i)}$. Consider the homomorphism $X_{\CC(S_i)}\to X_{\CC(S_1)\times \CC(S_2)}$ and compose it with $\Gamma_i'$ and denote the composition by $\Gamma_i'$. Suppose that for the generic points $\eta_i$ of $S_i$ we have
$$\Gamma_1'(\eta_1)=\Gamma_2'(\eta_2)\;.$$
Also assume  that $\Gamma_{1*},\Gamma_{2*}$ are both surjective, then there exists a correspondence $R$ from $S_1$ to $S_2$ such that $R$ induces a homomorphism from $A_0(S_1)$ to $A_0(S_2)$, and the kernel of $R$ is torsion if  $\Gamma_{1*}$ is injective.

\end{theorem}

\begin{proof}

Suppose that we are given that for the generic points of $S_1,S_2$, the correspondences coincide. Precisely it means the following. Let $\eta_1,\eta_2$ be two points of transcendence degree $2$ on $S_1,S_2$ respectively and $\Gamma'_1,\Gamma'_2$ denote the extension of $\Gamma_1,\Gamma_2$ over $\Spec(\CC(S_1)),\Spec(\CC(S_2))$, respectively, that is we have $\Gamma_i'$ supported on $S_{i\CC(S_i)}\times X_{\CC(S_i)}$. We further consider the morphism $X_{\CC(S_1)}\to X_{\CC(S_1)\times \CC(S_2)}$ and $X_{\CC(S_2)}\to X_{\CC(S_1)\times \CC(S_2)}$, consider the homomorphism induced by them at the level of Chow groups and demand that the image of $\Gamma'_1(\eta_1)$ is rationally equivalent to the image of $\Gamma'_2(\eta_2)$ on $ X_{\CC(S_1)\times \CC(S_2)}$. Now we know that Chow group of zero cycles of $X_{\CC(S_1)\times \CC(S_2)}$ is isomorphic to the colimit of Chow groups of zero cycles
$$\varinjlim \CH_0(X\times_{\Spec(k)}{U})$$
where $U$ is Zariski open in $S_1\times S_2$. So the above condition  means that there exists closed points $(s_1,s_2)$, in $U$ inside $S_1\times S_2$ such that
$$\Gamma_{1*}(s_1)=\Gamma_{2*}(s_2)\;.$$
Let us give some more details on it. We have $\CH_0(X\times_{\Spec(k)} U)$, is dominated by $\CH_0(X\times _{\Spec(k)}(U_1\times U_2))$, where $U_1,U_2$ are open in $S_1,S_2$, and since $\Gamma_{1*},\Gamma_{2*}$ are surjective when restricted to $U_1,U_2$. So we get the homomorphisms $\Gamma_{i*}$ from $\CH_0(U_i)$, to $\CH_0(X\times_{\Spec(k)} U_i)$, that will further surject onto $\CH_0(X\times _{\Spec(k)}U)$. Therefore the image of $\Gamma'_{i*}(\eta_i)$ for $i=1,2$ coincide means that there exists $z_1,z_2$ supported on $U_1,U_2$ such that $\Gamma_{1*}(z_1)$
is rationally equivalent to $\Gamma_{2*}(z_2)$. Since $\eta_i$'s are generic points of $S_i$, and we have that $\CH_0(\eta_i)=\varinjlim \CH_0(U_i)$, where $U_i$ is open in $S_i$, we have that $\eta_i$ is rationally equivalent to $s_i$, for some closed point $s_i$ on $U_i$. This actually shows that for any closed point $s_1$ in $U_1$, there exists $s_2$ in $U_2$ such that $\Gamma_{1*}(s_1)$ is rationally equivalent to $\Gamma_{2*}(s_2)$ and vice versa.

Then consider the set $R$ inside $S_{1}\times S_{2}$ given by the pairs $(s_1,s_2)$ such that $\Gamma_{1*}(s_1)=\Gamma_{2*}(s_2)$ in $A_1(X)$, by the above this set $R$ is non-empty. Assuming that $\Gamma_1,\Gamma_2$ are relative cycles in the sense of Suslin-Voevodsky \cite{SV}, we get a morphism from $S_{1}\times S_{2}\to C^d_1(X)\times C^d_1(X)$. Using this and the fact from \cite{R},\cite{M}, that the fibers of the map from $ C^d_1(X)\times C^d_1(X)$ to $A_1(X)$ are countable union of Zariski closed subsets of $ C^d_1(X)\times C^d_1(X)$, we get that $R$ is a countable union of Zariski closed subsets in $S_1\times S_2$. Since it maps surjectively onto $S_{1},S_{2}$, there exists a Zariski closed $R_1,R_2$, which map surjectively onto $S_{1},S_{2}$ respectively. Therefore $R_1,R_2$ are of dimension atleast $2$. Suppose that dimension of $R_1$ is $4$, then a general fiber of the projection from $R_1$ to $S_{1}$ is $2$, which gives us that $A_0(S_{2})$ is zero, therefore $A_1(X)$ is zero which is not true. So $R$ is of dimension $3$ or $2$, in the case when $R$ is of dimension $3$, in this case the fibers are of dimension $1$. For the generic point $\eta_1$ of $S_1$, the generic fiber $R_{1\eta_1}$ is a smooth projective curve, and its Jacobian maps to zero on $X$. Since $\CH_0(R_{1\eta_1})$ is the colimit of $\CH_0(R_{1U})$, where $U$ is Zariski open in $S_1$, we get that there exists a natural map from $\CH_0(R_1)$ to $\CH_0(R_{1\eta_1})$ and also we have $\CH_0(R_1)\to \CH_0(S_1)$, which is surjective by the construction of $R_1$. Since the colimit $\varinjlim \CH_0(R_{1U})$ maps onto the colimit $\varinjlim \CH_0(U)$, which is $\CH_0(\eta)$ and that again naturally maps to $\CH_0(S_1)$, so that $\CH_0(R_1)\to \CH_0(S)$ factors through $\CH_0(R_{1\eta_1})$ and hence $A_0(R_{1\eta_1})$ maps surjectively onto $A_0(S_1)$, so $A_0(S_1)$ becomes weakly representable, which is not true as it maps surjectively onto $A_1(X)$, which is non-representable. Hence we have a correspondence $R$ mapping finitely onto $S_1$, such that $R_*$ maps $A_0(S_1)$ to $A_0(S_2)$, by the following formula,
$$R_*(s_1)=\sum_i s_{1i}$$
where $\Gamma_{1*}(s_1)=\Gamma_{2*}(s_{1i})$ and we have that
$$\Gamma_{2*}R_*(s_1)=r\Gamma_{1*}(s_1)$$
where $r$ is the degree of the map from $R$ to $S_1$. In particular $R_*$ has torsion kernel.
So we proved the following theorem.

\end{proof}

\section{Kernel of a correspondence at the level of zero cycles on smooth projective surfaces}
In this section we are interested in the following problem, that let $R$ be a correspondence on $S_1\times S_2$, giving a homomorphism from $A_0(S_1)$ to $A_0(S_2)$. Can that  kernel of $R_*$ be the $A_0$ of a surface $S$. Precisely that means, can there exists a correspondence $\Gamma$ on $S\times S_1$, such that $\Gamma_*$ is an isomorphism onto the $\ker(R_*)$.  So we prove the following theorem :

\begin{theorem}
Suppose that $S_1$ has irregularity zero and the albanese map is not an isomorphism for both $S_1,S_2$ and supposing that $R_*$ is onto and not injective. Then there does not exist a smooth projective surface $S$ and a correspondence $\Gamma$ on $S\times S_1$ such that image of $\Gamma_*$ is equal to kernel of $R_*$ and $\Gamma_*$ isomorphism.
\end{theorem}
\begin{proof}
With our assumption we get the following things, that is the Jacobians of the hyperplane sections of $S_1$, generate $A_0(S_1)$, hence their images under $R_*$, generate $A_0(S_2)$. Since $A_0(S_1),A_0(S_2)$ are not weakly representable (that is not isomorphic to the albanese), the kernels of the corresponding homomorphisms from Jacobians of the hyperplane sections of $S_1$ to $A_0(S_1),A_0(S_2)$ are countable (this is due to the monodromy argument as in section \ref{section1}.

Let us embedd $S,S_1$ into some projective space. Let $S_t$ be a general smooth hyperplane section of $S$. We have $\Gamma$ supported on $S\times S_1$, of dimension $2$. Consider  the composition $\Gamma_*j_{t*}$ from $J(C_t)$ to $A_0(S_1)$. Since the image of this group is finite dimensional, there exists a hyperplane section of $S_1$ say $S_{1s}$ (we may have to choose a higher degree embedding of $S_1$ into a projective space) such  that
the image of $J(C_t)$ under $\Gamma_*j_{t*}$ is contained in $J(S_{1s})$. Hence by Mumford-Roitman type argument there exists a correspondence $\Gamma_{s,t}$ on $J(C_t)\times J(C_s)$ such that  we have
$$\Gamma_{s,t*}:J(S_{\eta})\to J(S_{1\eta})\;.$$
So for a general closed point $t$ and by the divisibility property of the group of algebraically trivial cycles modulo rational equivalence, we have the following commutative diagram.
$$
  \diagram
    J(S_{t})\ar[dd]_-{} \ar[rr]^-{\Gamma_{s,t*}} & & J(S_{1s}) \ar[dd]^-{} \\ \\
  A_0(S) \ar[rr]^-{} & & A_0(S_1)
  \enddiagram
  $$

Since $A_0(S)$ is not isomorphic to albanese of $S$, we have that the kernel of the left vertical homomorphism is countable for a general $t$. Since $\Gamma_*$ is injective so $\Gamma_{s,t*}$ has countable kernel. Now $J(S_t)$ goes to zero under $R_*$, so it goes to zero under $R_{t*}$. Now the kernel of $\Gamma_{s,t*},R_{t*}$ are countable therefore image of $\Gamma_{s,t*}$ is contained in the kernel of $R_*$, which is an uncountable set, this is a contradiction. So we cannot have $ker(R_*)$ isomorphic to $A_0(S)$.

\end{proof}

\section{The application of the above result for the study of one cycles on a cubic fourfold}

Let $X$ be a smooth cubic fourfold as before. Let $S$ be the discriminant surface of the projection from a line $l$ from $X$ to $\PR^3$. Let $\wt{S}$
be the double cover of $S$ inside the Fano variety of lines $F(X)$. Then as in the \ref{section1}, we have proven that the natural homomorphism from $A_0(\wt{S})$ to $A_1(X)$ is surjective. Also let $\pi$ be the $2:1$ map from $\wt{S}$ to $S$. Then we have a pull-back at the level of zero cycles
$$\pi^*:A_0(S)\to A_0(\wt{S})$$
then we have that
$$Z_*\pi^*(t-s)=Z_*(\pi^*(t)-\pi^*(s))=Z_*(l_s^1+l_s^2-l_t^1-l_t^2)=f^*(s-t)=0$$
where $f$ is the regular map obtained by blowing up the center of the projection from $X$ to $\PR^3$. Since $s-t$ is rationally equivalent to zero on $\PR^3$ we have that $f^*(s-t)$ is rationally equivalent to zero. So the kernel of $Z_*$ contains $\pi^*(A_0(S))$. Now we prove that the kernel of $Z_*$ is exactly $\pi^*(A_0(S))$.

\begin{theorem}
The group $A_1(X)$ is isomorphic to the  $A_0(\wt{S})/\pi^*( A_0(S))$.
\end{theorem}

\begin{proof}
So let $a$ belongs to the kernel of $Z_*$. Then $a$ is supported $\oplus A_0(\wt{S_t})$ for finitely many $t$, so we have that it is of the form $\sum_i Z_{t_i*}(a_{t_i})$, where $Z_{t*}$ is the homomorphism from $A_0(\wt{S_t})$ to $A_1(X)$. So to understand the kernel of $Z_*$, we have to understand the kernel of $\oplus_i Z_{t_i*}$ from
$$\oplus_{i}A_0(\wt{S_{t_i}})\to A_1(X)\;.$$
So let us consider the following commutative diagram.
$$
  \diagram
   \prod_i\Sym^g \wt{S_{t_i}}\ar[dd]_-{} \ar[rr]^-{} & & \bcC^d_1(X) \ar[dd]^-{} \\ \\
  \oplus_i A_0(\wt{S_{t_i}}) \ar[rr]^-{Z_{t_{i}*}} & & A_1(X)
  \enddiagram
  $$
Since the fiber of the right vertical morphism is a countable union of Zariski closed subsets in the Chow scheme of one cycles on the cubic, we have that the kernel of the lower horizontal homomorphism is a countable union of Zariski closed subsets in the direct sum
$$\oplus_i A_0(\wt{S_{t_i}})\;.$$
By using the uncountability of the ground field we get that the kernel is a countable union of shifts of an abelian variety $A$ inside the product
$\prod_i J(\wt{S_{t_i}})$. Since the abelian variety lies inside the product of Jacobians, it will follow that it is isogenous to the product of abelian varieties $A_i$, where each $A_i$ is embedded into $J(\wt{S_{t_i}})$. This fact comes from the equivalence between polarized Hodge structures and abelian varieties. The abelian variety $A$ corresponds to a Hodge structure in the direct sum $\oplus_i H^1(\wt{S_{t_i}},\QQ)$, so it decomposes into a sum $\oplus_i H_i$, where $H_i$ is a Hodge substructure in $H^1(\wt{S_{t_i}},\QQ)$. Now let $A$ be a proper non-trivial subvariety inside the product. Then there exists atleast one $A_i$, which is proper and non-trivial in $J(\wt{S_{t_i}})$. Let $D$ be a Lefschetz pencil through $X_{t_i}$. Then we have an action of $\pi_1(D\setminus 0_1,\cdots,0_m,t_i)$ on $H^3(X_{t_i},\QQ)$. So we have the induced action of the above fundamental group on $Z_{{t_i}*}(H^1(\wt{S_{t_i}},\QQ))$. So the image of  $A_i$ under $Z_{t_i*}$ gives rise to a Hodge substructure $H_i$ in $Z_{t_i*}(H^1(\wt{S_{t_i}},\QQ))$.

Now extend the scalars from $\CC$ to $\CC(t)$. Consider the abelian variety $A_{i\CC(t)}$ (by abuse of notation we mean the image of $A_i$ under $Z_{t_i*}$) over the function field $\CC(t)$. Let $L$ be a finite extension of $\CC(t)$ inside the algebraic closure $\bar{\CC(t)}$, such that $A_{i\CC(t)}, J(\wt{S_{t_i}})_{\CC(t)}$ are defined over $L$. Let $D'$ be a smooth projective curve which maps finitely onto $D$ and $\CC(D')=L$. Then we spread the abelian varieties $A_{i\CC(t)}, J(\wt{S_{t_i}})_{\CC(t)}$ over a Zariski open $U'$ in $D'$. Let us denote this spreads by $\bcA,\bcJ$ respectively. Throwing out some more points from $U'$ we get that the map from $\bcA,\bcJ$ to $U'$ is a proper submersion. So by the Ehressmann's fibration theorem we get a locally constant sheaf on which $\pi_1(U',t')$ acts. The locally constant sheaf has its stalk as $H^{2d-1}(A_{t_i},\QQ),H^1(\wt{S_{t_i}},\QQ)$ and $t'$ is a closed point of $U'$ lying over $t_i$. Now $\pi_(U',t')$ is a finite index subgroup in $\pi_1(U,t_i)$, where $U=D\setminus 0_1,\cdots, 0_m$. Now we prove that $\pi_1(U,t_i)$ acts  on $H^{2d-1}(A_{t_i},\QQ)$, that is this subspace is $\pi_1(U,t_i)$ stable. By the Picard Lefschetz formula we have that the action of $\pi_1(U,t_i)$ acting on $Z_{t_i*}H^1(\wt{S_{t_i}},\QQ)$ given by
$$\gamma. Z_{t_i*}(\alpha)=Z_{t_{i}*}(\alpha)\pm \langle Z_{t_{i}*}(\alpha),\delta_{\gamma} \rangle\delta_{\gamma}$$
where $Z_{t_i*}$ is the homomorphism of Hodge structures from $H^1(\wt{S_{t_i}},\QQ)$ to $H^3(X_{t_i},\QQ)$ and it is surjective, $\gamma$ is a generator of the fundamental group, $\delta_{\gamma}$ is the vanishing cycle corresponding to the generator $\gamma$. Since $\pi_1(U',t')$ is of finite index in $\pi_1(U,t_i)$, let this index be $m$. Then we have that
$$\gamma^m. \alpha=\alpha\pm m\langle \alpha,\delta_{\gamma}\rangle \delta_{\gamma}$$
where $\alpha$ belong to $H^{2d-1}(A_{t_i},\QQ)$. Since $H^{2d-1}(A_{t_i},\QQ)$ is $\pi_1(U',t')$ invariant and $\gamma^m$ belongs to $\pi_1(U',t')$, so we have that $\gamma^m.\alpha-\alpha$ belongs to $H^{2d-1}(A_{t_i},\QQ)$. Therefore we have that
$$m\langle \alpha,\delta_{\gamma}\rangle \delta_{\gamma}$$
is in $H^{2d-1}(A_{t_i},\QQ)$. Dividing by $m$ we get that $\langle \alpha,\delta_{\gamma}\rangle \delta_{\gamma}$ is in $H^{2d-1}(A_{t_i},\QQ)$, therefore again by Picard Lefschetz formula we have that
$\gamma.\alpha$ is in the subspace. So $H^{2d-1}(A_{t_i},\QQ)$ is $\pi_1(U,t_{i})$, stable, rather its image in $H^3(X_{t_i},\QQ)$ under $Z_{t_i*}$. So we get that either $H^{2d-1}(A_{t_i},\QQ)$ goes to zero under $Z_{t_i*}$ or it is all of $H^3(X_{t_i},\QQ)$. The second possibility says that $A_{t_i}$ maps surjectively to $A_1(X_{t_i})$, which if true for a general $t_i$, implies that $A_1(X)$, is representable, which is not true. Hence we have that $A_i$ is contained in the kernel of $Z_{t_i*}$ from $J(\wt{S_{t_i}})$ to $A_1(X_{t_i})$.

Now consider the cubics $X_t$ such that $l\subset X_t$. This is a Zariski closed subset in the parameter space of all smooth cubics. For such $X_t$, $S_{t_i}$ is the discriminant curve of the projection from $l$ onto $\PR^2$. So for such $X_t$, the kernel of $J(\wt{S_{t}})\to A_1(X_t)$ is $\pi^*(J(S_t))$, \cite{Be}[theorem 2.1 (iii)]. So the kernel is of dimension $g$ for a special but smooth $X_t$. But we can prove that the above map $J(\wt{S_t})\to A_1(X_t)$, for a smooth $X_t$, is a map of $\pi(U,t)$ modules, where $U$ parametrizes all smooth hyperplane sections of $X$. Therefore the dimension of the kernel is constant  (because the dimension of the kernel of  the map between corresponding local systems is constant). Therefore for a general $t$, the kernel $J(\wt{S_t})\to A_1(X_t)$ is of dimension $g$ and it contains $A_t$, hence also $\pi^* J(S_t)$. Therefore the kernel for a general $t$ is exactly equal to $A_t=\pi^*J(S_t)$.

This analysis tells us that $\prod_i A_i=A$ lies in the direct sum $\oplus_i \pi_i^*(J({S_{t_i}}))$, hence any element in the kernel of $Z_*$, is in $\pi^*(A_0({S}))$. Hence we have that $A_0(\wt{S})/\pi^*(A_0({S}))$ is isomorphic to $A_1(X)$.

\end{proof}

\begin{remark}

This analysis also tells us that $A_1(X)$ is the kernel of the push-forward from $A_0(\wt{S})$ to $A_0(S)$. Argument will be similar as above. By the previous section we know that, this kernel from $A_0(\wt{S})$ to $A_0(S)$ cannot be of the form $A_0(S')$ (since $A_0(\wt{S}),A_0(S)$ are not weakly representable), for some other surface $S'$, meaning that there  cannot exists a surface $S'$ and a correspondence supported on $S'\times \wt{S}$, such that $\Gamma_*$ is an isomorphism from $A_0(S')$ to  the kernel of $A_0(\wt{S})\to A_0(S)$.
\end{remark}

\section{Weak representability of Chow groups upto dimension two and its application to study 1-cycles on cubic fourfolds}

In the previous section we prove that the group $A_1(X)$ cannot be isomorphic $A_0(S)$, where $S$ is a smooth projective surface and $X$ is a cubic fourfold. In this section we weaken this notion of representability and we prove that there cannot exist a surjective homomorphism from $A_0(S)$ to $A_1(X)$, for some smooth projective surface $S$. For this we use the fact that $A_1(X)$ is isomorphic to the kernel of $A_0(S_1)\to A_0(S_2)$.

\begin{theorem}
\label{theorem3}
Let $S_1,S_2$ be two smooth projective surface with irregularity zero, such that $A_0(S_1),A_0(S_2)$ are not representable and $A_0(S_1)$ surjects onto $A_0(S_2)$ through a homomorphism induced by a correspondence and the homomorphism is not injective and the kernel is not weakly representable. Then there cannot exist a smooth projective surface $S$ such that and a correspondence $\Gamma$ on $S\times S_1$ such that $\Gamma_*$ is onto from $A_0(S)$ to $\ker(A_0(S_1)\to A_0(S_2))$.
\end{theorem}

\begin{proof}
We embed $S,S_1,S_2$, into some projective spaces. Consider a Lefschetz pencil on $S_1$ and on $S$. So we have a net on $S\times S_1$. The for a general member $(s,t)$ of $\PR^1\times \PR^1$, we have $J(S_s)$ mapping to $J(S_{1t})$. Let $\Gamma_{s,t}$ be the correspondence inducing this homomorphism. This would mean that we have a morphism of Hodge structures $\Gamma_{s,t*}$ from $H^1(S_s,\QQ)$ to $H^1(S_{1t},\QQ)$. Now we observe that image of $\Gamma_{s,t*}$ is $\pi_1(\PR^1\setminus \{0_1,\cdots,0_m\},t)$ stable. This is because of the following reason. The Jacobian $J(S_s)$ is mapping  into $J(S_{1t})$. So it gives rise to an abelian subvariety $A_t$ inside $J(S_{1t})$. Now this abelian subvariety, by extending scalars is defined over $\Spec(\CC(x))$, where $\CC(x)$ is the function field of the affine line. So attaching coefficients of the defining equations of $A_t,J(S_{1t})$, they are defined over a finite extension $L$ of $\CC(x)$. Then find a smooth projective curve $D'$ mapping finitely onto $\PR^1$, and having the function field equal to $L$. Then we spread the two abelian varieties $A_t, J(S_{1t})$ over a Zariski open $U$ in $D'$, throwing out more points we get that this two spreads give rise to a fibration over $U$, hence we have $\pi_1(U,t')$ acting on $H^1(S_{1t},\QQ)$ and on $im(\Gamma_{s,t*})$. Then it is a consequence of the Picard Lefschetz formula that image of $\Gamma_{s,t*}$ is $\pi_1(\PR^1\setminus \{0_1,\cdots,0_m\},t)$ stable. So by the irreducibility of the monodromy action we have that image of $\Gamma_{s,t*}$ is either zero or all of $H^1(S_{1t},\QQ)$. The first option is not possible since $\Gamma_*$ is surjective and onto the kernel $\ker(A_0(S_1)\to A_0(S_2))$, and this kernel is not weakly representable.  So the only option is that $\Gamma_{s,t*}$ is surjective. So suppose that the kernel of $J(S_s)\to A_0(S_1)$ is a countable union of translates of an abelian variety $A_s$ of $J(S_s)$. Then  the image of $A_s$ under $\Gamma_{s,t*}$ is giving rise to a $\pi_1(\PR^1\setminus \{0_1,\cdots,0_m\},t)$ stable subspace of $H^1(S_t,\QQ)$. So it is either $0$ or all of $H^1(S_t,\QQ)$. In the second case $J(S_{1t})$ maps to zero for a general $t$, which is a contradiction to the fact that $A_0(S_1)$ is non-representable hence the kernel of $J(S_{1t})\to A_0(S_1)$ is countable for a general $t$. So the second case is possible that is $A_s$ lies in the kernel of $J(S_s)\to J(S_{1t})$ and it is a proper abelian subvariety of $J(S_s)$ since $J(S_s)$ does not map to zero under $\Gamma_{s,t*}$. So we have that $J(S_s)$ modulo the kernel of $J(S_s)\to A_0(S_1)$ is inside the kernel of $J(S_{1t})\to A_0(S_2)$, which is countable since $A_0(S_1)\to A_0(S_2)$ is surjective and $A_0(S_2)$ is non-representable. But the above analysis says that it contains a uncountable set namely $J(S_s)$ modulo the kernel of $J(S_s)\to A_0(S_1)$, because otherwise $J(S_s)$ will be a countable union of its proper Zariski closed subsets. This is a contradiction, so there cannot exists $S$ and a correspondence $\Gamma$, such that $\Gamma_*$ from $A_0(S)$ to the kernel of $A_0(S_1)\to A_0(S_2)$ is surjective.
\end{proof}

\begin{definition}
Let $X$ be a fourfold. We say that $A_1(X)$ is weakly representable upto dimension $2$, if there exists finite many curves $C_i$, finitely many surfaces $S_j$ and correspondences $\Gamma_i$, $\Gamma_j$, such that
$$\oplus \Gamma_{i*}\oplus \Gamma_{j*}:\oplus J(C_i)\oplus A_0(S_j)\to A_1(X)$$
is onto.
\end{definition}

Now we prove the following.

\begin{theorem}
 Weak representability upto dimension $2$ is a birational invariant, meaning that if $X$ is birational to $Y$, then $A_1(X)$ is weakly representable upto dimension $2$ if and only if it is so for $A_1(Y)$.
\end{theorem}

\begin{proof}
Let $X$ be birational to $Y$. Then the indeterminacy locus of the birational map $X\dashrightarrow Y$ is either of codimension $2$ or greater. So that if we blow up $X$ along the indeterminacy locus we get $\wt{X}$ and the Chow group $A_1(\wt{X})$ is isomorphic to $A_1(X)\oplus J(C)$ or $A_1(X)\oplus A_0(S)\oplus A^1(S)$, depending on whether the indeterminacy locus is a curve or a surface. So if $A_1(X)$ is weakly representable upto dimension $2$, then so is $A_1(\wt{X})$. Since $\wt{X}$ maps surjectively onto $Y$, we have that $A_1(Y)$ is weakly representable upto dimension $2$. Similarly changing the role of $X$ and $Y$ we get that $A_1(Y)$ is weakly representable upto dimension $2$ implies that it is so for $A_1(X)$.
\end{proof}

Similar argument as in previous theorem \ref{theorem3} will tell us that for the kernel of $A_0(S_1)\to A_0(S_2)$ where $A_0(S_i)$ is non-representable for $i=1,2$, there cannot exist a surjection from $A_0(S)\oplus J(C)$, for some smooth projective curve $C$ and a smooth projective surface $S$. This means that we have the following.

\begin{theorem}
\label{theorem4}
If $X$ is a cubic fourfold birational to $\PR^4$, then $X$ cannot be obtained as a single blow up followed by a single blow down.
\end{theorem}

\subsection{The example of the $A_1$ of a cubic fourfold containing a plane}

Let $X$ be a cubic fourfold containing a plane $P$. Then we project from $P$, onto $\PR^2$, so that the restriction of this projection onto $X$ is a rational map. So blowing up the indeterminacy locus of this rational map, which is $P$, we get $\wt{X_P}$, which inherits a quadric bundle structure over $\PR^2$. Suppose that this quadric bundle satisfies the assumption of proposition 2.3 in \cite{H1}. Let us consider the projection from a point on this quadric bundle onto $\PR^4$, it is known that this map is birational. So the quadric bundle is rational. Now by the previous theorem \ref{theorem4}, we get that the quadric bundle cannot be obtained by one blow up along a curve or surface in $\PR^4$ followed by one blow down. In that case the blow up of the cubic is weakly representable upto dimension $2$ and not only that, since we obtain $\wt{X_P}$ by blowing up a plane in $X$, we have $A_1(X)$ isomorphic to $A_1(\wt{X_P})$. So we get that $A_1(X)$ is isomorphic to $A_0(S)\oplus \Pic(S)$ for some surface $S$ in $\PR^4$ or isomorphic to $J(C)$ for some smooth projective curve $C$ in $\PR^4$, which is not possible. Therefore the quadric bundle $\wt{X_P}$ is obtained by more than one blowing up of $\PR^4$, followed by a sequence of blow downs. Since $\PR^4$ is weakly representable upto dimension $2$, it will follow that the quadric bundle is weakly representable upto dimension $2$.

\end{document}